\documentclass{amsart}
\usepackage{amsmath,amsfonts,amssymb,amsthm}

\usepackage{amsmath}

\newtheorem{teo}{Theorem}
\newtheorem{lem}{Lemma}[section]
\newtheorem{pro}[lem]{Proposition}
\newtheorem{cor}[lem]{Corollary}
\newtheorem{quo}[lem]{Question}
\newtheorem{rem}[lem]{Remark}

\newcommand{\argu}{\hbox to 7truept{\hrulefill}}

\DeclareMathOperator{\im}{im}

\DeclareMathOperator{\Aut}{Aut}

\DeclareMathOperator{\Der}{Der}

\DeclareMathOperator{\Inn}{Inn}

\newcommand{\Z}{\mathbb{Z}}

\newcommand{\N}{\mathbb{N}}

\newcommand{\Q}{\mathbb{Q}}

\date{10/06/2014}
 
\begin{document}

\title{Finite $p$-groups with small automorphism group}
 \author{Jon Gonz\'alez-S\'anchez}
\address{Departamento de Matem\'aticas, 
  Facultad de Ciencia y Tecnolog\'ia,
  Universidad del Pa\'is Vasco- Euskal Herriko Unibertsitatea,
  Apartado 644 48080 Bilbao, Spain}

\author {Andrei Jaikin-Zapirain}
\address{Departamento de Matem\'aticas,
  Universidad Aut\'onoma de Madrid, and Instituto de Ciencias Matem\'aticas, CSIC-UAM-UC3M-UCM,  28049-Madrid, Spain}


\maketitle

\begin{abstract}
For each prime $p$ we construct a family $\{G_i\}$ of finite $p$-groups such that $|\Aut (G_i)|/|G_i|$ goes to $0$, as $i$ goes to infinity. This disproves a well-known conjecture that  $|G|$ divides $|\Aut(G)|$ for every non-abelian finite $p$-group $G$.
\end{abstract}


\section{Introduction}

A well-known question (see, for example, \cite[Problem 12.77]{Ko}) asks whether it is true that $|G|$ divides $|\Aut(G)|$ for every non-abelian finite $p$-group $G$. It is not clear who raised this question first explicitly, the first result in this direction that we have found in the literature is due to E. Schenkman \cite{Sch} and it is more than 50 years old. In that paper E. Schenkman showed that this is true for finite non-abelian $p$-groups of class 2 (the proof has a gap, which is corrected  by R. Faudree in \cite{Fa}).  Later it  was also  established for $p$-groups of exponent $p$ in \cite{Re}, for  $p$-groups of maximal class in \cite{Ot},  for $p$-groups with center of order $p$ in \cite{Gas},  
      for metacyclic $p$-groups when $p$ is odd in \cite{Da3},  for central-by-metacyclic $p$-groups when $p$ is odd in \cite{D02}, for $p$-abelian
$p$-groups in \cite{Da} (see also \cite{Th}),  for finite modular $p$-groups in \cite{DO},   for some central products in \cite{Hu, Bu}, for $p$-groups with center of index at
most $p^4$ in \cite{Da2}, for $p$á-groups with cyclic Frattini subgroup in \cite{Ex2},  for $p$-groups of order at most
$p^6$ in \cite{Da2, Ex}, for $p$-groups of order at most
$p^7$ in \cite{Gav}, for  $p$-groups of coclass 2 in \cite{FJO} (see also a related result in \cite{Ei}), for $p$-groups $G$ such that $(G,Z(G))$ is a Camina pair in \cite{Ya}.

All these partial results indicate that a counterexample to the problem should have a large size and it will be difficult to present it explicitly. In this paper we use the pro-$p$ techniques and we are able to show the following.

\begin{teo} \label{main} For each prime $p$ there exists  a family of finite $p$-groups $\{U_i\}$ such that 
$$\lim_{i\to \infty} |U_i|=\infty \textrm{\ and \ } \limsup_{i\to \infty} \frac{|\Aut U_i|}{|U_i|^{\frac{40}{41}}}<\infty.$$
In particular, for every prime $p$, there exists a non-abelian finite $p$-group $G$ such that $|\Aut(G)| < |G|$.
\end{teo}
Let us briefly explain our construction. It consists of two parts:
\begin{enumerate}
\item  Firstly, we take an infinite finitely generated  pro-$p$ group $U$ such that $\Aut (U)$ is ``smaller" than $U$. Of course, to be smaller for infinite groups does not refer to the order. In our construction $U$ will be a uniform $p$-adic pro-$p$ group  and so we  can speak about $\dim U$. Recall that $\dim U$ is defined as $\dim_{\Q_p} \textbf{L} (U)$, where  $\mathbf{L}(U)$ is the Lie $\Q_p$-algebra associated with $U$. Since $U$ is  compact $p$-adic analytic, $\Aut (U)$ is  also a $p$-adic analytic profinite group. Thus, $\Aut U$ is ``smaller" than $U$ will  simply mean that $\dim \Aut(U)<\dim (U)$. 

\item Secondly, $U$ can be written as an inverse limit   $U= \varprojlim U_i$  of finite $p$-groups $U_i=U/N_i$, where $N_i=U^{p^i}$.
Since $\Aut(U)= \varprojlim \Aut(U_i)$, we may hope that $\Aut(U_i)$ are smaller than $U_i$ when $i$ is large (compare with Lemma \ref{quotients}). 
\end{enumerate}
In order to construct $U$ from the first step, we notice that if $U$ is a uniform pro-$p$ group, then $$  \dim \Aut (U)=\dim_{\Q_p}(\Der(\mathbf{L}(U)),$$ where  $\Der(\mathbf{L}(U))$ is the algebra of $\Q_p$-derivations of $\textbf L(U)$. The examples of   Lie algebras $L$ with $\dim\Der(L)<\dim L$ are known to exist and were first constructed by E. Luks \cite{Lu} and T. Sato \cite{Sa}. In Sato's example the algebra is constructed over $\Q$, it has dimension 41, its center has dimension 1 and its derived algebra consists only of inner derivations (and so, it has dimension 40). This is the explanation for the numbers which appear in Theorem \ref{main}.

The realization of the second step of our proof is based on an analysis of the first cohomology groups $H^1(U,L_i)$, where $L_i=\textbf{log}(U)/p^i\textbf{log}(U)$ and $\textbf{log}(U)$ is the Lie ring corresponding to a uniform pro-$p$ group $U$ by Lazard's correspondence. 
It turns out that since $\Der(\mathbf L(U))=\Inn (\textbf L(U))$, $ \Der(\textbf {log} (U))$ is finite, and so, $$H_{cts}^1(U,\textbf {log}(U))\cong \Der(\textbf {log} (U))$$ is finite.  This implies the existence of a uniform upper bound for $|H^1(U,L_i)|$. As consequence, we obtain a uniform upper bound for $|\Aut(U_i): \Inn(U_i)|$, where $U_i=U/U^{p^i}$, that finishes the proof.

The organization of the paper is as follows. In Section \ref{uniform} we describe the basic  facts about $p$-adic analytic groups, introduce continuous  cohomology groups of of pro-$p$ groups and establish the uniform upper bound for $|H^1(U,L_i)|$. In Section \ref{proof} we present the proof of Theorem \ref{main}.

\section{Uniform pro-$p$ groups and their cohomology groups}

\label{uniform}
\subsection{Uniform pro-$p$ groups}

Let $L$ be a Lie $\Z_p$-algebra. We say that $L$ is {\it uniform} if for some $k$, 
  $L\cong \Z_p^k$ as $\Z_p$-module and
 $[L,L]\subseteq 2pL$. 
Analogously, we say that a pro-$p$ group $U$ is {\it uniform} if 
 it is torsion-free, finitely generated and $[G,G]\subseteq G^{2p}$. 
 
 One can define the functors $\textbf{exp}$ and $\textbf{log}$ between the categories of uniform Lie $\Z_p$ -algebras and uniform pro-$p$-groups which are isomorphism of categories (see \cite[Section 4]{DDMS}).  There is a relatively easy way to define the functor $\mathbf {log}$. If $U$ is a uniform pro-$p$ group, then $\mathbf {log}(U)$ is the Lie $\Z_p$-algebra, which underlying set coincides with $U$ and the Lie operations are defined as follows 
\begin{equation}\label{lieoperations} a+b=\lim_{i\to \infty} (a^{p^{i}}b^{p^{i}})^{1/p^{i}},\ [a,b]_L=\lim_{i\to \infty}  [a^{p^{i}},b^{p^{i}}]^{1/p^{2i}},\ a,b\in U.\end{equation}
 If   $f:U\to V$ is a homomorphism between two uniform pro-$p$ groups, then $\mathbf {log} (f)=f$ is a homomorphism of Lie $\Z_p$-algebras.  In particular, the conjugation converts $\mathbf{log}(U)$ in $U$-module.
 
 \begin{lem} \label{isom}   Let $U$ be a uniform pro-$p$ group. Let $i,j\in \N$ be such that $i\le j\le 2i+1$. Then $U^{p^i}/U^{p^j}$ is abelian and 
 $$U^{p^i}/U^{p^j}\cong \mathbf{log} (U)/p^{j-i}\mathbf {log}(U)$$ as $U$-modules  ($U$ acts on $U^{p^{i}}/U^{p^{j}}$ by conjugation).
  \end{lem}
    \begin{proof}
    The lemma is a consequence of the definition of sum in   (\ref{lieoperations}).
    \end{proof}
  Let $G$ be a  $p$-adic analytic profinite group. Then it contains a uniform open subgroup $U$.  The Lie algebra $\mathbf{L}(G)$ of $G$ is a Lie $\Q_p$-algebra defined as $\mathbf{L}(G)=\textbf{log} (U)\otimes_{\Z_p } \Q_p$. The definition does not depend on the choice of $U$. We put $\dim G=\dim_{\Q_p} \mathbf{L}(G)$. For a $p$-adic pro-$p$ group we have the following internal characterization of the dimension:
\begin{lem} (\cite[Proposition III.3.1.8]{La}, \cite[Lemma 4.10]{DDMS})  \label{quotients} Let $G$ be a $p$-adic analytic pro-$p$ group. Denote by $G_i=G^{p^i}$ the subgroup of $G$ generated by $p^i$th powers of the elements of $G$. Then there are positive constants $c_1$ and $c_2$ such that 
$$ c_1p^{i\dim G}\le |G:G_i|\le c_2 p^{i\dim G}.$$
Moreover, if $G$ is uniform $|G:G_i|=p^{i\dim G}$.
\end{lem}

\label{cohomology}

\subsection{Continuous cohomologies of pro-$p$ groups} In this subsection we  present the results about continuous cohomology groups of pro-$p$ groups that we will need in this paper. More details and omitted proofs can be found in  \cite{NSW, SW}.

 Let $G$ be a pro-$p$ group. We say that $A$ is a {\it topological}  $G$-module if $A$ is an abelian Hausdorff topological group which is endowed with the structure of an abstract left $G$-module such that the action $G\times A\to A$ is continuous. In this paper, $A$ will be one the following three types: a finite abelian group, a profinite abelian group  and  a finitely dimensional  vector space over $\Q_p$. We denote by $G^{(i)}$ the cartesian product of $i$ copies of $G$. We put
$$\mathcal C_{cts}^i(G,A)=\{ f: G^{(i)}\to A\mid \text{ $f$ is a continuous functionÊ}\}$$
and denote the  coboundary operator $\partial_A^{i+1}:\mathcal C_{cts}^{i}(G,A)\to \mathcal C_{cts}^{(i+1)}(G,A)$ by means of 
$$\begin{array}{lll}

(\partial_A^{i+1} f)(g_1,\ldots, g_{i+1}) & = & g_1\cdot f(g_2,\ldots,g_{i+1}) \\
&&\\ &&+\sum_{j=1}^i (-1)^j f(g_1,\ldots,g_{j-1},g_jg_{j+1},g_{j+2},\ldots,g_{i+1})\\ &&\\ &&+(-1)^{i+1}f(g_1,\ldots, g_i).\end{array}$$
Now, we set 
$$
\mathcal Z_{cts}^i(G,A) =\ker \partial_A^{i+1}\textrm{\ and \
  }  \mathcal B_{cts}^i(G,A) =\im \partial_A^i$$
and define  the {\it  $i$th continuous cohomology group}
  $H_{cts}^i(G,A)$ of $G$ with coefficients in $A$  by $$
H_{cts}^i(G,A)= \mathcal Z_{cts}^i(G,A) / \mathcal B_{cts}^i(G,A).$$
If $A$ is a finite $p$-group, then $H_{cts}^i(G,A)$ coincides with the usual definition of $H_{}^i(G,A)$ and it is equal to
$\mathrm{Ext}^i_{\Z_p[[G]]}(\Z_p,A)
$ (see \cite[Chapter 3.2]{SW}).

If $\alpha: A\to B$ is a continuous homomorphism of topological $G$-modules, then we have the induced homomorphism of complexes
$$\tilde \alpha: \mathcal (\mathcal C_{cts}^* (G,A), \partial) \to \mathcal (\mathcal C_{cts}^*(G,B), \partial),\ (\tilde \alpha f)(g_1,\ldots,g_i)=\alpha( f(g_1,\ldots, g_i)).$$
 Hence, $\tilde \alpha$ extends to the homomorphisms of the homology groups of this complexes $\alpha_i^*: H_{cts}^i(G,A)\to H_{cts}^i(G,B)$.
 
 By \cite[Lemma 2.7.2]{NSW}, we have the following long exact sequence in cohomologies:
\begin{lem} \label{longexact}
Let $G$ be a  pro-$p$ group and 
$$0\to A\xrightarrow{\alpha} B\xrightarrow{\beta} C\to 0$$ be a short exact sequence  left   $\Z_p[[G]]$-modules with $C$ finite. Then there exists a canonical boundary homomorphism 
$$\delta : H^1_{cts}(G,C)\to H^{2}_{cts}(G,A)$$ such that

$$H^1_{cts}(G,A)\xrightarrow{\alpha_1^*}H^1_{cts}(G,B)\xrightarrow{\beta_1^*}  H_{cts}^1(G,C)\xrightarrow{\delta} H_{cts}^2(G,A)\xrightarrow{\alpha_2^*} H_{cts}^2(G, B)$$ is exact.
\end{lem}

 We say that $G$ is of type  $FP_{\infty}$ if
the trivial  $\Z_p[[G]]$-module $\Z_p$ has a  
free resolution over $\Z_p[[G]]$ such that all free modules are finitely generated. For example, if $G$ is $p$-adic analytic, then $\Z_p[[G]]$ is Noetherian (\cite[Proposition V.2.2.4]{La}, \cite[Corollary 7.25]{DDMS}), and so, $G$ is $FP_{\infty}$.

In the case when  $G$ is a  $FP_\infty$ pro-$p$ group and $A$  is  a topological pro-$p$  $G$-module,  $H_{cts}^i(G,A)$   coincides with $\mathrm{Ext}^i_{\Z_p[[G]]}(\Z_p,A)$ (see \cite[Theorem 3.7.2]{SW}).  Hence, if $A$ is finitely generated  as a $\Z_p$-module, then  $H_{cts}^i(G,A)$ are also finitely generated as $\Z_p$-modules.

\subsection{The first cohomology groups of a uniform group}\label{u}
In this subsection we consider a uniform pro-$p$ group $U$ such that $\textbf{L}(U)$ has only inner derivations and try to understand  its first cohomology groups with coefficients in some natural modules. First we consider $H_{cts}^1(U,\mathbf{log} (U))$. 
\begin{pro}\label{boundinfinite}  Let $U$ be a uniform pro-$p$ group. Assume that the Lie algebra $\textbf{L}(U)$ has only inner derivations. Then  $H_{cts}^1(U,\mathbf{log} (U))$ is finite.
\end{pro}  

\begin{proof} Since $U$ is a finitely generated pro-$p$ group and $\mathbf{log}(U)$ is  finitely generated as a $\Z_p$-module, $H_{cts}^1(U,\mathbf{log}(U))$ is also finitely generated as $\Z_p$-module.
 Hence it is enough to show that $H_{cts}^1(U,\mathbf{log}(U))$ is a torsion module, i.e. $$H_{cts}^1(U,\mathbf{log}(U))\otimes_{\Z_p} \Q_p$$ is equal to 0. 
 By \cite[Theorem 3.8.2]{SW}, $$H_{cts}^1(U,\mathbf{log}(U))\otimes_{\Z_p} \Q_p=H_{cts}^1(U,\textbf{L}(U))$$ and by \cite[Theorem 5.2.4]{SW}, 
 $$H_{cts}^1(U,\textbf{L}(U))\cong H^1(\textbf L(U),\textbf{L}(U)).$$ Note that by definition of $H^1(\textbf L(U),\textbf L(U))$,  $$H^1(\textbf L(U),\textbf L(U))=Der( \textbf L(U))/\Inn(\textbf L(U))$$
 and by our hypotheses, it is equal to zero.
  \end{proof}
 \begin{rem} There is an alternative way to prove the previous proposition. One can show directly that $H_{cts}^1(U,\mathbf{log}(U))\cong \Der(\mathbf{log}(U))/\Inn(\mathbf{log}(U))$ and conclude that, since  $\textbf{L}(U)$ has only inner derivations, $\Der(\mathbf{log}(U))/\Inn(\mathbf{log}(U))$  is finite.
 \end{rem}
 Now, we can bound   $|H^1_{cts}(U,\mathbf{log}(U)/p^i\mathbf{log} (U))|$ uniformly in $i$.
\begin{pro} \label{boundfinite} Let $U$ be a uniform pro-$p$ group. Assume that the Lie algebra $\textbf{L}(U)$ has only inner derivations. Then there exists a constant $C$ such that $$|H^1_{cts}(U,\mathbf{log}(U)/p^i\mathbf{log} (U))|\le C$$ for every $i$.
\end{pro}
  \begin{proof} Let $\alpha: \mathbf{log}(U)\to \mathbf{log}(U)$ be a multiplication by $p^i$. Then $$\alpha_2^*:H_{cts}^2 (U,\mathbf{log}(U))\to H_{cts}^2 (U,\mathbf{log}(U))$$ is also the multiplication by $p^i$. Hence $\ker \alpha_2^*$ is contained in the torsion part of $H_{cts}^2 (U,\mathbf{log}(U))$. 
  
  Since $U$ is $FP_\infty$ and $\mathbf{log}(U)$ is finitely generated as a $\Z_p$-module, we have  that $H_{cts}^2 (U,\mathbf{log}(U))$ is also finitely generated as a $\Z_p$-module. Hence its torsion part is finite. Applying Lemma \ref{longexact}, we conclude that $$|H^1_{cts}(U,\mathbf{log}(U)/p^i\mathbf{log} (U))|\le |H_{cts}^1(U,\mathbf{log}(U))||Torsion(H_{cts}^2(U,\mathbf{log}(U)))|.$$
  Thus, Proposition \ref{boundinfinite} implies Proposition \ref{boundfinite}.
   \end{proof}
 
   \section{Proof of Theorem \ref{main}} \label{proof}

The next example is the basement of our construction.
\begin{pro} (\cite{Sa}) There exists a Lie $\Q$-algebra $M$ of dimension 41 such that $\dim Z(M)=1$ and $ Der(M)$ consists only of inner derivations. 
\end{pro}
This proposition allows us to construct  uniform pro-$p$ groups considered in Subsection \ref{u}. It is done in the following way.

The algebra $M$ has a subring $M_0$, such that $M=M_0\otimes_\Z \Q$. Let $L=p^2(M_0\otimes_\Z \Z_p )$. Then $L$ is a uniform Lie $\Z_p$-algebra. If we put $U=\mathbf{exp}(L)$, then $L=\mathbf{log}(U)$ and $\textbf L(U)\cong L\otimes_{\Z_p} \Q_p\cong M\otimes_\Q \Q_p$. 

\begin{lem} The Lie $\Q_p$-algebra $\textbf L(U)$ is of dimension 41, its center  has dimension 1 and $ Der(\textbf L(U))$ consists only of inner derivations. \end{lem}
\begin{proof} This easily follows from the fact $\textbf L(U)\cong  M\otimes_\Q \Q_p$.   \end{proof}

Let $U_i=U/U^{p^i}$.
Denote by $\rho_{i,j}: \Aut(U_i)\to \Aut (U_j)$ (for $i\ge j$)  the map
$$\rho_{i,j}(\alpha  )(uU^{p^j})=\alpha(uU^{p^i})U^{p^j}, \textrm{\ for \ } \alpha\in \Aut(U_i),\ u\in U.$$

Now, we are ready to present the main step in our proof.
\begin{pro} There exists a constant $k$ such that for all $i\ge 2k$,  $$\ker \rho_{i,k}\le \Inn(U_i)\ker\rho_ {i,i-k}.$$
\end{pro}
\begin{proof} By Proposition \ref{boundfinite}, there exists $k$ such that $$p^kH^1_{cts}(U,\mathbf{log}(U)/p^i\mathbf{log} (U))=0$$ for all $i$. We will prove the proposition by induction on $i$. When $i=2k$, the proposition is clear. Assume that we have shown the proposition for $i$, let us prove  it for $i+1$. 

Let $\phi \in \ker \rho_{i+1,k}$. Since $\rho_{i+1,i}(\phi)\in \ker \rho_{i,k}$, the inductive assumption implies that $\phi\in \ker \rho_{i+1,i-k}\Inn U_{i+1}$. Thus, without loss of generality we may assume that $\phi\in \ker \rho_{i+1,i-k}$. 

Define the following function $c:U\to U^{p^{i-k}}/U^{p^{i+1}}$:
$$c(u)=\phi(uU^{p^{i+1}})u^{-1}.$$
Then $$c(u_1u_2)=\alpha(u_1u_2U^{p^{i+1}})u_2^{-1}u_1^{-1}=c(u_1) u_1c(u_2)u_1^{-1}.$$
Thus $c\in \mathcal Z_{cts}^1(U,  U^{p^{i-k}}/U^{p^{i+1}})$. 

By Lemma \ref{isom}, $U^{p^{i-k}}/U^{p^{i+1}}$ is abelian and it is isomorphic to $\mathbf{log}(U)/p^{k+1}\mathbf{log}(U)$ as an $U$-module. In particular, $$p^{k}H^1_{cts}(U,  U^{p^{i-k}}/U^{p^{i+1}})=0.$$ 
Consider the following exact sequence of $U$-modules:
$$1\to U^{p^{i-k+1}}/U^{p^{i+1}}\xrightarrow{\alpha} U^{p^{i-k}}/U^{p^{i+1}}\xrightarrow{\beta} U^{p^{i}}/U^{p^{i+1}}\to 1,$$
where $\alpha$ is the inclusion and $\beta$ is the $p^k$th power map. Then $\beta_1^*$ is multiplication by $p^k$ and so by Lemma \ref{longexact}, 
$\im \alpha_1^*=H^1_{cts}(U,  U^{p^{i-k}}/U^{p^{i+1}})$. Hence there are  $c^\prime \in  \mathcal Z_{cts}^1(U,  U^{p^{i-k+1}}/U^{p^{i+1}})$
 and $v\in U^{p^{i-k}}/U^{p^{i+1}}$ such that
$$c(u)=c^\prime(u) vuv^{-1}u^{-1} \textrm{\ for every \ } u\in U.$$
Thus, we obtain that
$$\phi(uU^{p^{i+1}})=c(u) u=c^\prime(u) vuv^{-1} =c^\prime(u) vuU^{p^{i+1}}v^{-1}\textrm{\ for every \ } u\in U.$$
But this implies that $\phi\in \ker \rho_{i+1,i+1-k}\Inn (U_{i+1})$ and we are done.
 \end{proof}

\begin{cor} \label{inner}There exists a constant $D$ such that $$|\Aut(U_i):\Inn(U_i)|\le D \textrm{\ for all \ } i.$$
\end{cor}
\begin{proof} By the previous proposition we have that
\begin{eqnarray*}
|\Aut(U_i):\Inn(U_i)|& \le & |\Aut(U_i): \Inn(U_i)\ker \rho_{i,i-k}|| \Inn(U_i)\ker \rho_{i,i-k}:\Inn(U_i)| \\
& \leq &|\Aut(U_i):\ker \rho_{i,k}||\ker \rho_{i,i-k}|\le |\Aut(U_k)||\ker \rho_{i,i-k}|.
\end{eqnarray*}
Since the number of generators of $U_i$ is 41 and $|U^{p^{i-k}}/U^{p^i}|=p^{41k}$, 
we obtain that $|\ker \rho_{i,i-k}|\le p^{(41)^2k}$. This finishes the proof of the corollary.
 \end{proof}

Now we are ready to prove Theorem \ref{main}

\begin{proof}[Proof of Theorem \ref{main}]  By Lemma \ref{quotients}, $$|U_i|=p^{41i}.$$ Since $Z(U)$ is 1-dimensional,  $\dim U/Z(U)=40$. Hence
$$|\Inn (U_i)|\le |U/U^{p^i}Z(U)|=p^{40 i}.$$
Now, the theorem follows from Corollary \ref{inner}
 \end{proof}

\section{Final remarks}
Let $\phi$ be the Euler totient function. It is not difficult to show that for a finite abelian group $A$, $|\Aut(A)|\ge \phi(|A|)$. In \cite[Problem 15.43]{Ko} Deaconescu has asked if the same is true for an arbitrary finite group. 
The examples from \cite{BW1,BW2} show that $\frac{|\Aut(G)|}{\phi(|G|)}$ can be made arbitrary small when $G$ is a soluble or perfect finite group. Our examples show that in fact $$\frac{|\Aut(G)|}{\phi(|G|)^{\frac{40}{41}+\epsilon}}$$ 
can be made arbitrary small for every $\epsilon>0$ when $G$ is a finite nilpotent group. 

This also provides a counterexample to a conjecture from \cite{BW2} that says that for a finite non-nilpotent supersoluble group $G$, $|\Aut(G)|>\phi(|G|)$. For this, simply take a family $\{U_i\}$ of $5$-groups from Theorem \ref{main} and consider the following family of finite non-nilpotent supersoluble groups $\{\Sigma_3\times U_i\}$.

As a consequence of the previous discussion we would like to raise the following question.
\begin{quo} Does there exist a constant $\alpha>0$ such that for every finite group $|G|$
$$|\Aut(G)|\ge \phi(|G|)^\alpha?$$
\end{quo}
By a classical result of  W. Ledermann and B. H. Neumann \cite{LN}, there exists a function
$g(h)$ having the property that $| \Aut(G)|_p\ge p^h$ whenever $| G|_p \ge p^{g(h)}$. A quadratic upper bound for $g$ was established by  J. A. Green \cite{Gr}  and until now only not very important improvements of Green's bound have been obtained (see, for example, \cite{Hy}).

\

{\bf Acknowledgement.}
 This paper is partially supported by the   grant MTM 2011-28229-C02-01 of the Spanish MEyC and by the   ICMAT Severo Ochoa project SEV-2011-0087.  J.~Gonz\'alez-S\'anchez also acknowledges support
  through the Ram\'on y Cajal Programme of the Spanish Ministry of
  Science and Innovation.

We would like to thank Rutwig Campoamor Stursberg  for answering our questions about Lie algebras and for providing useful references and Anitha Thillaisundaram for pointing out several typos in the first version of the article.

\end{document}